\documentclass[11pt,leqno,a4paper]{amsart}
\usepackage{amscd,amssymb}
\usepackage{verbatim}
\usepackage[all]{xy}
\newtheorem{Thm}{Theorem}[section]
\newtheorem{Lem}[Thm]{Lemma}

\newtheorem{Prop}[Thm]{Proposition}
\newtheorem{Not}[Thm]{Notation}
\theoremstyle{definition}
\newtheorem{Def}[Thm]{Definition}
\theoremstyle{remark}
\newtheorem{Rem}{Remark}
\theoremstyle{example}

\newcommand{\Hom}{\mathrm{Hom}}
\newcommand{\Cogen}{\mathrm{Cogen\ }}
\newcommand{\Cone}{\mathrm{Cone}}
\newcommand{\Gen}{\mathrm{Gen}}
\newcommand{\Prod}{\mathrm{Prod\ }}
\newcommand{\Add}{\mathrm{Add\ }}
 \newcommand{\Ext}{\mathrm{Ext}}
 \newcommand{\End}{\mathrm{End}}
\newcommand{\Tor}{\mathrm{Tor}}
 \newcommand{\Ker}{\mathrm{Ker}}
  
\newcommand{\Rej}{\mathrm{Rej}}

\newcommand{\Mod}{\mbox{\rm Mod-}}

\newcommand{\LMod}{\mbox{\rm -Mod}}

\newcommand*{\cclass}[1]{\mathcal{#1}}
\newcommand*{\F}{\cclass{F}}
\newcommand*{\M}{\cclass{M}}
\newcommand*{\A}{\cclass{A}}

\newcommand*{\C}{\cclass{C}}

\newcommand*{\T}{\cclass{T}}

\newcommand*{\D}{\cclass{D}}
\newcommand*{\X}{\cclass{X}}
\newcommand*{\Y}{\cclass{Y}}
\newcommand*{\E}{\cclass{E}}

\newcommand*{\R}{\mathbb{R}}
\newcommand*{\N}{\mathbb{N}}
\newcommand*{\LL}{\mathbb{L}}
\newcommand*{\ES}{\cclass{S}}

\def\doublearrow#1#2{\mathrel{\mathop{\vcenter{\offinterlineskip
\hbox{$\longrightarrow$} \hbox{$\longleftarrow$}}}\limits_{#2}^{#1}}}

\def\to{\rightarrow}

\begin{document}
%Topmatter
\title[Equivalences induced by tilting modules]{Equivalences induced by  infinitely generated tilting modules}
\author{Silvana Bazzoni}
\address
{Dipartimento di Matematica Pura e Applicata, \\ Universit\`a di
Padova\\ Via Trieste 63, 35121 Padova, Italy}
\email{bazzoni@math.unipd.it}  
\thanks {Supported by MIUR, PRIN 2005, project ``Perspectives in the theory of rings, Hopf algebras and categories of modules'' and by Universit\`{a} di Padova (Progetto di Ateneo
CPDA071244/07 `` 
Algebras and cluster categories'' ).
 \protect\newline 2000 Mathematics
Subject Classification. Primary: 16D90; 16E30; 18E30; Secondary:  16S90; 16G10.
 \protect\newline Key words: tilting modules, equivalences and derived equivalences.}
\begin{abstract}  We generalize Brenner and Butler's Theorem as well as Happel's Theorem on the equivalences induced by a finitely generated tilting module over artin algebras, to the case of an infinitely generated tilting module over an arbitrary associative ring establishing the equivalences  induced between subcategories of module categories and also at the level of derived categories.
\end{abstract}
\maketitle

\section{Introduction} 

Tilting theory started in the context of finitely generated modules over artin algebras and was further generalized over arbitrary associative rings with unit and to infinitely generated modules (see  \cite{CF}, \cite{CT}, \cite{CDT}, \cite{ATT}).

One of the most important features in classical tilting theory is the famous Brenner and Butler's Theorem~\cite{BB} establishing two equivalences between suitable categories of finitely generated modules. 

A finitely generated tilting module $T$ over an artin algebra $\Lambda$ gives rise to a torsion pair  $(\T,\F)$, where  $\T$ is the class of modules generated by $T$. If $D$ denotes the standard duality and $\Gamma$ is the endomorphism ring of $T$, then $D(T)$ is a cotilting $\Gamma$-modules with an associated torsion pair $(\X,\Y)$ where $\Y$ is the class modules cogenerated by $D(T)$. The Brenner and Butler's Theorem states that the functor $\Hom_{\Lambda}(T,-)$ induces an equivalence between the categories $\T$ and $\Y$ with inverse the functor $-\otimes_{\Gamma}T$, and the functor $\Ext^1_{\Lambda}(T,-)$ induces an equivalence between $\F$ and $\X$ with inverse the functor $\Tor^{\Gamma}_1(-, T)$. (See \cite{HR} and \cite{HRS}).

Moreover, $\T$ is the kernel of the functor $\Ext^1_{\Lambda}(T,-)$, $\Y$ is the kernel of $ \Tor^{\Gamma}_1(-,T)$,  $\F$ is the kernel of $ \Hom_{\Lambda}(T, -)$ and $\X$ is the kernel of $-\otimes_{\Gamma}T$.

Later on, Happel~\cite{H} observed that the natural context in which to interpret the above equivalences is that of derived categories.
He proved that the total right derived functor of the functor $\Hom_{\Lambda}(T,-)$ induces a derived equivalence between the bounded derived categories of finitely generated $\Lambda$-modules and the bounded derived categories of finitely generated $\Gamma$-modules.

Colby and Fuller~\cite{CF}  proved  a ``Tilting Theorem''  for finitely presented tilting modules over an arbitrary associative ring, generalizing  Brenner and Butler's Theorem, and Colpi~\cite{C} extended the Tilting Theorem to the wider context of Grothedieck categories.

%In the present paper we consider the case of infinitely generated tilting modules and we prove a generalization of Brenner-Butler's Theorem as well as Happel's Theorem for derived categories.

The first instance of a generalization of Brenner and Butler's Theorem to infinitely generated tilting modules, appears in two papers by Facchini~\cite{F1}, ~\cite{F2} where he studied the equivalences induced by the tilting module $\partial$, a divisible module introduced by Fuchs ~\cite{F} over commutative domains.
The theorems proved by Facchini provide a link between the Brenner and Butler tilting equivalences and the equivalences established by Harrison and Matlis between subcategories of modules over a commutative domain $R$.

If $Q$ is the quotient filed of a commutative domain $R$ and $K$ is the module $Q/R$, then Harrison and Matlis' Theorem states that the functor $\Hom_R(K,-)$ induces an equivalence between the category of $h$-divisible torsion modules and the category of torsion free cotorsion modules. Moreover, the functor $\Ext^1(K, -)$ gives an equivalence between the category of $h$-reduced torsion $R$-modules and the category of special cotorsion modules. %
Thus the similarity with tilting equivalences was evident and the papers by Facchini showed the advantage to work with a tilting module, namely the module $\partial$ rather than the module $K$, even though the formal definition of an infinitely generated tilting module was not yet available.

In this paper we generalize both Brenner and Butler Theorem's and Facchini results, to the case of an arbitrary (infinitely generated) tilting module over an associative ring $R$. If $\Mod R$ is the category of all right $R$-modules  and $T\in \Mod R$ is a tilting module, $T$ induces a torsion pair $(\T, \F)$ in $\Mod R$, where $\T$ is the class of modules genrated by $T$. If $S$ is the endomorphism ring of $T$, we prove that the dual $T^d$ of $T$ with respect to an injective cogenerator of $\Mod R$, is a partial cotilting right $S$-module inducing a torsion pair $(\T_{T^d}, \F_{T^d})$ in $\Mod S$.

By Theorem~\ref{T:equivalences}, we prove that the functor $\Hom_R(T,-)$ induces an equivalence between the category $\T$ and the intersection of $\F_{T^d}$ with a suitable subcategory $\M$ of $\Mod S$, namely the double perpendicular category of the module $T^d$ (see definition in Section~\ref{subclasses}).
Secondly, the functor $\Ext^1_R(T,-)$  induces an equivalence between $\F$ and the intersection of $\T_{T^d}$ with the subcategory $\M$.
Moreover, the inverses of these equivalences are given by the functors $-\otimes_ST$ and $\Tor^S_1(-, T)$.

The subcategories of $\Mod S$ equivalent to $\T$ and $\F$ in the above equivalences cannot be interpreted as Gabriel quotients of $\Mod S$, since there are no Serre subcategories arising in the process.
Thus again, as in the case of  finitely generated tilting modules, the situation can be better illustrated in the context of derived categories, where the equivalences involved can be formulated in a concise and more expressive way. In fact, if $\D(R)$ and $\D(S)$ are the (unbounded) derived categories of the categories $\Mod R$ and $\Mod S$ respectively, we prove that the total right derived functor of the functor $\Hom_R(T,-)$, that is the functor $\R \Hom_R(T, -)$, induces an equivalence between $\D(R)$ and the quotient category of $\D(S)$ modulo the full triangulated subcategory $\Ker(-\overset{\LL}\otimes _ST)$, namely the kernel of the total left derived functor of the functor $-\otimes_ST$.

{\bf Acknowledgement} I wish to thank Bernhard Keller for his help in proving a crucial step in the proof of Lemma~\ref{L:counit} and Pedro Nicol\'as for suggesting the use of the powerful Proposition 1.3 in Gabriel Zisman's book
~\cite{GZ}.
\section{Preliminaries}
In what follows all rings are associative with unit. We recall some definitions and results.

For a ring $R$, $\Mod R$ ($R \LMod$)  will denote the category of all right (left)
$R$-modules.
% and by $\rmod R$, 
%the category of all right compact modules, that is the modules possessing a projective resolution consisting of
% finitely generated projective modules. (If $R$ is a right
%coherent ring then $\rmod R$ is just the category of all
%finitely presented modules.)

%For $n \in \omega$, 
For an $R$-module $M$ we denote by %$\PP_1$  the class of all modules of projective dimension $\leq
%1 $, and by $\I_1$ the class of all modules of injective dimension $\leq
%1 $. Moreover, 
p.d. $M$ and i.d. $M$  the projective and injective dimension of $M$, respectively.

If $\lambda$ is a cardinal, $M^{(\lambda)}$ and $M^{\lambda}$ will denote the direct sum and the direct product of $\lambda$ copies of $M$, respectively.

Let $\mathcal{C} \subseteq\Mod R$. Define 

\[\C ^{\perp}=\{X\in \Mod R \mid \Ext_R^i(C, X)=0 \ {\rm for \ all}\ C\in \C,  \ {\rm for \ all\ } i\geq1 \},\]
\[ ^{\perp}\C=\{X\in \Mod R \mid \Ext_R^i(X, C)=0 \ {\rm for \ all}\ C\in \C,  \ {\rm for \ all\ } i\geq1 \}.\]

\begin{Def} {\rm (\cite{CT}, \cite{ATT})} An $R$-module $T$ is \emph{$1$-tilting} provided \newline
{\rm (T1)} {\rm p.d.}$T\leq 1$, \newline
{\rm (T2)} $\Ext^i_R(T, T^{(\lambda)}) =0$ for each $i \geq 1$ and every cardinal $\lambda$, and
\newline
{\rm (T3)} there exists an exact sequence

$$0 \rightarrow R \rightarrow T_0 \rightarrow T_1 \rightarrow 0$$

such that $T_i \in\Add T$ for each $0 \leq i \leq 1$.

Here, $\Add T$ denotes the class of all direct summands of arbitrary direct
sums of copies of $T$.

If $T$ is an $1$-tilting module, $T^{\perp }$ is called $1$\emph{-tilting} class.\end{Def}

%\begin{Rem} {\emph 
%If $T$ is an $1$-tilting module, then  $T^\perp$ is a torsion class and $^\perp(T^\perp)\subseteq \PP_1$.
%}\end{Rem}
%
\begin{Def} {\rm (\cite{CT})} 
An $R$-module  $T$ is \emph{$1$-partial tilting} if $T$ satisfies {\rm (T1)}, {\rm (T2)} and $T^\perp$ is closed under direct sums.
\end{Def}

We have also dual definitions.
\begin{Def} {\rm (\cite{CDT}, \cite{ATT})} A module $C$ is \emph{$1$-cotilting} provided \newline
{\rm (C1)} {\rm i.d.}$C \leq1$, \newline
{\rm (C2)} $\Ext^i_R(C^{\lambda}, C) =0$ for each $i \geq 1$ and every cardinal $\lambda$, and
\newline
{\rm (C3)} there exists an exact sequence

$$0 \rightarrow C_1 \rightarrow C_0 \rightarrow W \rightarrow 0$$

such that $C_i \in \Prod C$ for each $0 \leq i \leq 1$ and $W$ is an injective $R$-cogenerator.

Here, $\Prod C$ denotes the class of all direct summands of arbitrary direct
products of copies of $C$.

If $C$ is an $1$-cotilting module, $^{\perp }C$ is called $1$\emph{-cotilting} class.
\end{Def}  

%\begin{Rem} {\emph 
%If $C$ is an $1$-cotilting module, then  $^\perp C$ is a torsion free class and $(^\perp C)^\perp\subseteq \I_1$.
%}\end{Rem}

\begin{Def}{\rm (\cite{CDT})}
An $R$-module  $C$ is \emph{$1$-partial cotilting} if $C$ satisfies {\rm (C1)}, {\rm (C2)} and $^\perp C$ is closed under direct products.
\end{Def}

If $T$ and $U$ are $1$-tilting ($1$-cotilting) modules, then $T$ is
\emph{equivalent} to $U$ if $T^{\perp }=U^{\perp }$ ($^{\perp }T={}^{\perp}U$), which is the case if
and only if $\Add T=\Add U$ ($\Prod T=\Prod U$).

%\begin{Def} A class $\C$ of modules is said to be of \emph{finite type},
%provided there is a set $\mathcal{S}\subseteq \PP_n$ of modules with projective resolution consisting of finitely generated projective modules, such that $%
%\C= \mathcal{S} ^\perp$.
%\end{Def}

We recall some results on infinitely generated $1$-tilting and $1$-cotilting modules which give a better understanding of their properties.
\begin{itemize}

\item By \cite[1.3]{CT} a module $T$ is $1$-tilting if and only if
 $T^{\perp }=\Gen T$, where $\Gen T$ is the class of modules generated by $T$.

\item  By ~\cite{BH} If $T$ is a $1$-tilting module, then the tilting class $T^\perp$ is of finite type, that is there is a set $\ES$ of finitely presented modules of projective dimension at most $1$, such that $\ES^\perp= T^\perp$.

\item By  ~\cite{B} $1$-cotilting modules are pure injective.

\item As a consequence of the above results, we have that every $1$-tilting right $R$-module $T$ induces a torsion pair $(\T, \F)$ in $\Mod R$ where $\T=\Gen T=T^\perp$ and $\F= \Ker (\Hom_R(T,-)).$\\
Every $1$-cotilting right $R$-module $C$ induces a torsion pair $(\T, \F)$ in $\Mod R$ where $\F=\Cogen C={}^\perp C$ and $\T=\Ker (\Hom_R(-, C))$. Moreover, the cotilting torsion free class $\F$ is closed under epimorphic images.

\end{itemize}
\section{Infinitely generated $1$-tilting modules.}\label{phi-and-e}
In this section we adapt the results proved by Facchini in \cite{F1} and \cite{F2} for the case of the tilting module $\partial$ defined over a commutative domain, to the case of a tilting module over  an arbitrary associative ring.

First of all we have to make a suitable choice of a representative in the equivalence class of a $1$-tilting module.
\begin{Prop}\label{P:simple-sequence} Let $R$ be a ring and let
$T_R$ be a $1$-tilting module. Up to equivalence we can assume that $T$ fits in an exact sequence of the form:
\[0\to R\to T\to T_1\to 0\]
where $ T_1$ is a direct summand of $T$. \end{Prop}
\begin{proof} From condition {\rm (T3}) in the definition of tilting modules, we have an exact sequence
$$0 \to R\overset{\iota} \to T_0 \to T_ 1\to 0$$
where $T_0, T_1 \in\Add T$. Consider the module $T'=T_0\oplus (T_1)^{(\omega)}$ and let $j\colon T_0\to T'$ be the natural embedding of $T_0$ in $T'$. Then we have an exact sequence:
$$ 0 \to R\overset{j \circ \iota} \to T'\to T_1\oplus(T_1)^{(\omega)}\to 0$$
where $T_1\oplus(T_1)^{(\omega)}\cong (T_1)^{(\omega)}$ is isomorphic to a direct summand of $T'$. 
Thus we also have an exact sequence
$$ 0 \to R \to T'\to T'_1\to 0$$
with $T'_1$ a direct summand of $T'$.
Now $T'$ is a $1$-tilting module. In fact, $T'$ satisfies conditions {\rm (T1)} and {\rm (T2)} since $T'\in \Add T$; it satisfies also {\rm (T3)}, by the above sequence. Moreover, $T'$ and $T$ are equivalent, since $T^\perp\subset T'^\perp$ and $T'^\perp =\Gen T'\subseteq \Gen T=T^\perp.$
\end{proof}

\begin{Not}\label{N:notation}
\emph{From now on we assume that $T$ is a $1$-tilting right $R$-module such that the short exact sequence of condition {\rm (T3)} has the form
\[\rm{(a)} \qquad 0\to R \overset{\mu}\to T\to  T_1\to 0\]
where $T_1$ is a direct summand of $T$.
Moreover, we denote by $S$ the endomorphism ring of $T$.}

\emph{As in \cite{F1} we fix the following notations:
\begin{enumerate}
\item $\mu(1_R)=w\in T$.
\item %If $\iota$ is the embedding of $T_1$ into $T$ we let $\phi=\sigma\circ\iota$. Thus 
$\phi$ is an endomorphism of $T$ such that $\Ker\phi=wR$ and $\phi (T)$ is a direct summand of $T$.
\item $e$ is a fixed idempotent endomorphism of $T$ such that $e(T)=\phi (T).$
\end{enumerate}}
\end{Not}

\begin{Lem}\label{L:S-sequence} Let $T$ be as in Notation~\ref{N:notation}. There is a short exact sequence of left $S$-modules
\[\rm{(b)}\qquad 0\to I \to S\to  T\to 0\] such that
\begin{enumerate}
\item [(1)] $I$ is the left ideal $\{f\in S\mid f(w)=0\}$ and also $I=S\phi$;
\item [(2)] $I$ is isomorphic to $Se$;
\item[(3)] $\End_S(T)\cong R$;
\item [(4)]  $_ST$ is a cyclically presented partial $1$-tilting $S$-module.
\end{enumerate}
\end{Lem}

\begin{proof} (1) and (2) follow by applying the functor $\Hom_R(-, T)$ to the exact sequence (a).

So $_ST$ is cyclically presented; (3) and (4) follow by \cite[Lemma 2.15]{CT}.
\end{proof} 

\begin{Prop}[\cite{CT}, ~\cite{F2}]\label{P:Hom-Tor} Let $T$ be a $1$-tilting right $R$-module as in Notation~\ref{N:notation}. 
The following hold:
\begin{enumerate}
\item[(1)] The natural homomorphism (the counit of the adjunction) \[\phi\colon \Hom_R(T, M)\otimes_ST\to M\] is an isomorphism  if and only if $M$ in the tilting class $T^\perp.$
\item[(2)] $\Tor_1^S(\Hom_R(T, M), T)=0$, for every right $R$-module $M$.
\end{enumerate}
\end{Prop}
\begin{proof} (1) is proved in \cite[Corollary 2.18]{CT}.

(2) The proof is the same as in \cite[Proposition 4.2]{F1}, but we repeat the argument because our context is different. Let $N$ be a right $S$-module; applying the functor $N\otimes_S-$ to the exact sequence (b), we get that $\Tor_1^S(N, T)$ is the kernel of the map $N\otimes_SI\to N$. Since $I=S\phi\cong Se$ we have that $\Tor_1^S(N, T)$ is isomorphic to the kernel of the abelian group morphism $Ne\to N$ defined by $xe\mapsto x\phi$, for every $x\in N$, hence $\Tor_1^S(N, T)$ is isomorphic to $\{x\in N\mid x\phi=0\}e$.

So we need to show that, for every right $R$-module $M$, if $Y=\{g\in \Hom_R(T, M)\mid g\phi=0\}$, then $Ye=0$. Now $g\phi=0$ if and only if $\Ker g\supseteq \phi(T)=eT$ if and only if $ge=0$.
\end{proof}

\section{Equivalences between subclasses of modules}\label{subclasses}

For every right $R$-module $M$ we denote by $M^d$ the dual of $M$ with respect to an injective cogenerator $W$ od $\Mod R$, that is $M^d=\Hom_R(M, W)$. 
\begin{Prop}\label{P:partial-cotilting} Let the assumption be as in Notation~\ref{N:notation}. The right $S$-module $T^d$ satisfies the following properties:
\begin{enumerate}

\item[(1)] $[\Tor_i^S(-, T)]^d\cong \Ext^i_S(-, T^d)$. In particular, i.d.$(T^d)_S\leq 1$.
\item[(2)] $\Tor_1^S(T^d, T)\cong[\Ext^1_S(T, T)]^d=0$.
\item[(3)] $T^d$ is a partial $1$-cotilting right $S$-module.
\end{enumerate}
\end{Prop}

\begin{proof}  % (1) Is proved in \cite[Proposition 2.3]{AHT}
(1) Follows by a well known Ext-Tor relation and i.d.$(T^d)_S\leq 1$, since p.d.$_ST\leq 1$, so $\Tor_2^S(-, T)=0$.

(2)  $_ST$ is a finitely presented left $S$-module and p.d.$_ST\leq 1$, hence $\Tor_1^S(T^d, T)\cong[\Ext^1_S(T, T)]^d$ and $\Ext^1_S(T, T)=0$, since $_ST$ is a partial $1$-tilting module. 

(3) By (1) i.d.$(T^d)_S\leq 1$. Let $\{N_i\}_i$ be a family of right $S$-modules such that $\Ext_S^1(N_i, T^d)=0$, for every $i$. By (1) we have $\Tor_1^S(N_i, T)=0$. We show that $\Ext^1_S(\prod_iN_i, T^d)=0$. By (1) again, we have $\Ext^1_S(\prod_iN_i, T^d)\cong [\Tor_1^S(\prod_iN_i, T)]^d$. Since $_ST$ is finitely presented and p.d.$_ST\leq 1$, $\Tor_1^S(-, T)$ commutes with direct products, hence $\Ext^1_S(\prod_iN_i, T^d)\cong [\prod_i(\Tor_1^S(N_i, T)]^d$. But, as noted above,  $\Tor_1^S(N_i, T)=0$. Thus,  $^\perp T^d$ is closed under direct products. To conclude that $T^d$ is a partial $1$-cotilting module, it is enough to check that $\Ext^1_S(T^d, T^d)=0$. Now, $\Ext^1_S(T^d, T^d)\cong [\Tor_1^S(T^d, T)]^d$ and $\Tor_1^S(T^d, T)=0$ by (2). 
 \end{proof}

Recall that if $M$ is an $R$-module over a ring $R$, the preradical $\Rej_M$ is the subfunctor of the identity functor defined by $\Rej_M(X)= \cap \Ker\{f\mid f\in \Hom_R(X,M)\}$, for every $R$-module $X$. $\Rej_M$ is always a radical and if it is also idempotent, then it is a torsion radical (see from \cite{Bo}). In this case the associated torsion class consists of the modules $X$ such that $\Hom_R(X,M)=0$ and the torsion free class is $\Cogen M$.

\begin{Prop}\label{P:torsionfree-classes} In the same notations as Proposition~\ref{P:partial-cotilting} the partial $1$-cotilting $S$-module $T^d$ satisfies the following conditions:
\begin{enumerate}
\item [(1)] $T^d_S$ is a direct summand of a $1$-cotilting right $S$-module $C$ such that $^\perp C={}^\perp T^d$.
%\item[(1)] There is an exact sequence $0\to (T^d)^{\lambda}\to D \to E\to 0$ for a cardinal $\lambda$ and an injective $S$-cogenerator $E$ such that $T^d\oplus D=C$ is a $1$-cotilting right $S$-module with $^\perp C=^\perp T^d$.
\item[(2)] The preradical $\Rej_{T^d}$ is an idempotent radical inducing a torsion pair $(\T_{T^d}, \F_{T^d})$ in $\Mod S$ where $ \F_{T^d}= \Cogen T^d\subseteq {}^\perp{}T^d$.
\end{enumerate}
\end{Prop}

\begin{proof}
(1) See \cite[Theorem 2.11]{CTT}.

(2) Follows by \cite[Lemma 2.6]{CTT} and $\Cogen\  T^d\subseteq {}^\perp T^d$, since $T^d$ is a partial $1$-cotilting module. 
\end{proof}

 We define now the subcategories of $\Mod S$ which will play a crucial role in establishing the equivalences which will be proved by Theorem~\ref{T:equivalences}.
 
 First, we recall the notion of  perpendicular categories. If $\C$ is a category of $R$-modules the \emph{right perpendicular category} $\C_{\perp}$  is
 
\[\C_{\perp}=\{M\mid \Hom_R(\C, M)=\Ext^i_R(\C, M)=0\}\]
and analogously the \emph{left perpendicular category}  is
\[_{\perp}\C=\{M\mid \Hom_R(M, \C)=\Ext^i_R(M, \C)=0\}\]

We will also use the following definitions.
 \begin{Def} Let $\C$ be a subcategory of an abelian category $\A$.
\begin{enumerate}
\item[(1)] $\C$ has the $2$ out of $3$ property if for every short exact sequence
\[0\to L\to M\to N\to 0\]
in $\A$ with two terms in $\C$, then the third term is also in $\C$.
\item[(2)] $\C$ is a Serre subcategory if for every short exact sequence
\[0\to L\to M\to N\to 0\] in $\A$, $M$ is in $\C$ if and only if $L$ and $N$ are in $\C$.
\end{enumerate}
\end{Def}

\begin{Prop}\label{P:the class-E}. Let 
\[\E=\{N\in \Mod S\mid N\otimes_ST=\Tor_1^S(N, T)=0\}\]
The following hold:
\begin{enumerate}
\item[(1)] $\E=\{N\in \Mod S\mid \Ext^1_S(N,T^d)=\Hom_S(N, T^d)=0\}$, that is $\E={}_\perp \{T^d\}$.
\item[(2)] $\E$ is closed under direct sums, direct summands and has the $2$ out of $3$ property. 
%\item[(3)] $\E= \F_D\cap \T_{T^d}$.
\end{enumerate}
\end{Prop}

\begin{proof} (1) The equality follows by  the usual homological formulas and by the fact that i.d.$T^d\leq 1$. 

(2) Follows by a direct check.
\end{proof}%

Consider now the right perpendicular category $\M$ of $\E$, that is
\[\M=\E_{\perp}=\{M\in \Mod S\mid \Hom_S(\E, M)=0=\Ext^1_S(\E, M) \}.
\]

The next theorem, inspired by Facchini's Theorems, is the generalization of the equivalences proved by Brenner and Bluter ~\cite{BB} in the case of a classical $1$-tilting module (that is finitely generated) over artin algebras.

%The proposition can be proved as in \cite{F1} and \cite{F2}.
\begin{Thm}(\cite{F1}, \cite{F2})\label{T:equivalences} Let $R$ be a ring, $T_R$ a $1$-tilting module as in Notation~\ref{N:notation} and let $(\T,\F)$ be the tilting torsion pair associated to $T$. Let $S=\End_R(T)$.
The following hold.
\begin{enumerate}
\item[(1)] There is an equivalence 
\[\Mod R\supseteq \T\overset{\Hom_R(T, -)}\longrightarrow\Y\subseteq \Mod S\]
where $\Y=\F_{T^d}\cap \M$ with inverse $-\otimes_ST$.
 \item[(2)]  There is an equivalence 
\[\Mod R\supseteq \F\overset{\Ext^1_R(T, -)}\longrightarrow\X\subseteq \Mod S\]
where $\X=\T_{T^d}\cap \M$ with inverse $\Tor_1^S(-, T)$.
\end{enumerate}
\end{Thm}
\begin{proof}  The proof of the two equivalences is essentially the same as \cite{F1} and \cite{F2} with the suitable translation of the terminology.

In those papers the modules named $I$-divisible are the modules in $\T_{T^d}$, that is the right $S$-modules $N$ such that $\Hom_S(N, T^d)=$ or equivalently, $N\otimes_ST=0$. The modules called $I$-reduced are the modules in $\F_{T^d}$.
Moreover, the modules in the class $\E$ are called $I$-divisible and $I$-torsion free.

(1) Is proved by the same arguments as in \cite{F1}, once it is observed that the modules named $I$-cotorsion in that paper are the modules in the class $\F_{T^d}\cap \M$.

First one shows as in \cite[Theorem 7.1]{F1}, that for every $M\in \Mod R$, $\Hom_R(T,M)\in \F_{T^d}\cap \M$. 
Then one uses that, by Proposition~\ref{P:Hom-Tor}~(1), $\phi\colon \Hom_R(T, M)\otimes_ST\to M$ is an isomorphism  if and only if $M$ in the tilting torsion class $\T$.
Finally one verifies that  $\eta  \colon N\to\Hom_R(T, N\otimes_ST)$ is an isomorphism if and only if $N$ is a right $S$-module in the class $ \F_{T^d}\cap \M$. This is obtained following the proofs of \cite[Theorem 7.2, 7.3]{F1}.

(2) This is proved as in \cite{F2} noticing that there, slightly differently from the definitions in \cite{F1}, the modules named $I$-cotorsion are the modules in the class $ \M$.

First, as in \cite[Lemma 1]{F2}, one proves that $\Ext^1_R(T, M)$ is in the class $\T_{T^d}\cap \M$, for every right $R$-module $M$.
Secondly one shows that, if $M$ in the torsion free class  $\F$, then the natural homomorphism $\xi\colon \Tor^S_1(\Ext^1_R(T, M), T)\to M$ is an isomorphism (see \cite[Lemma 1]{F2}).

Then, one proves that $ \Tor^S_1(N,T)\in \F$, for every $N\in\T_{T^d}\cap \M$ and that the natural homomorphism $\theta\colon N\to \Ext^1_R(T, \Tor^S_1(N,T))$ is an isomorphism if and only if $N\in\T_{T^d}\cap \M$ (see \cite[Lemma 2]{F2}).

 \end{proof}
\begin{Rem}
 If $T$ is a finitely presented $1$-tilting module, then the dual module $T^d$  is a $1$-cotilting module over the endomorphism ring of $T$.
Hence, in this case, the category $\E={}_\perp T^d$ is zero, so $\M$  coincides with $\Mod S$ and $(\X, \Y)$ is the cotilting torsion pair associated to the $1$-cotilting module $T^d$. Thus, we recover both Brenner and Butler's  Theroem for the case of artin algebras and Colby-Fuller Tilting theorem over an arbitrary ring. So, Theorem~\ref{T:equivalences} can be viewed as the generalization to the case of infinitely generated $1$-tilting modules of Brenner and Butler's and Colby-Fuller's Theorems. 
 
\end{Rem}

The categories $\Ker( -\otimes_ST)$ and $\Ker(\Tor_1^S(-,T)$ are not Serre subcategories of $\Mod S$ in general. Thus, we cannot perform the corresponding quotient categories in Gabriel sense. However, we can localize the category $\Mod S$ at a suitable multiplicative system as we are going to explain.

In the next proposition we use the terminology as in the Gabriel and Zisman's book \cite{GZ}.
\begin{Prop}\label{P:calculus-of-fractions} Let $T$ be a $1$ tilting right $R$-module as in Notation~\ref{N:notation} and let $(\T,\F)$ be the associated torsion pair in $\Mod R$. Let $\Sigma$ be the system of morphisms $u\in \Mod S$ such that $u\otimes_S1_T$ is invertible in $\Mod R$.
Then the following hold:
\begin{enumerate}
\item[(1)] $\Sigma$ admits a calculus of left fractions.
\item[(2)] There is an equivalence $\rho\colon \Mod S[\Sigma^{-1}]\to \T$ such that $ \rho\circ q=-\otimes_ST$ where $q\colon \Mod S\to  \Mod S[\Sigma^{-1}]$ is the canonical localization functor.
\item[(3)] There is an equivalence between  $\Mod S[\Sigma^{-1}]$ and the category $\Y=\F_{T^d}\cap \M$.
\end{enumerate}
\end{Prop}
\begin{proof} (1) Note that $N\otimes_ST\in \T$ for every right $S$-module $N$ and $\T$ is a full subcategory of $\Mod R$. Hence the functor $H=\Hom_R(T,-)\colon\T\to \Mod S$ is right adjoint to the functor $G=-\otimes_ST\colon \Mod S\to \T$.
By Proposition~\ref{P:Hom-Tor}~(1) the counit adjunction $\phi\colon GH\to 1_{\T}  $ is invertible and, by \cite[Proposition 1.3]{GZ}, $H$ is a fully faithful functor. Hence $\Sigma$ admits a calculus of left fraction by \cite[2.5~(b)]{GZ}.

(2) Follows by Proposition~\ref{P:Hom-Tor}~(1) and by \cite[Proposition 1.3]{GZ}.

(3) Combine (2) with  Theorem~\ref{T:equivalences}(1).
\end{proof}

\begin{Rem} We couldn't get an analogous result  for the pair of functors $\Ext^1_R(T,-)$ and $\Tor_1^S(-, T)$ because they are not an adjoint pair in general.

 Moreover, we don't know wether the category of fractions $\Mod S[\Sigma^{-1}]$, considered in Proposition~\ref{P:calculus-of-fractions}, is the quotient of $\Mod S$ modulo a suitable subcategory.
\end{Rem}

The above remark indicate that a better understanding of the whole situation can be obtained in the setting of derived categories.

\section{Derived equivalence}

Before stating the main result of this section we recall some notions and facts about derived categories which will be used later on.

Let $\D(R)$ and $\D(S)$ be the unbounded derived categories of $\Mod R$  and $\Mod S$ respectively. The following hold.

\begin{itemize}
\item (B\"okstedt and Neeman~\cite{BN} or Spaltenstein~\cite{S}) For every complex $M^{\cdot}\in \D(R)$ there is a quasi isomorphism $M^{\cdot}\to I^{\cdot}$ where  $ I^{\cdot}$ is a complex with injective terms. $ I^{\cdot}$ is also denoted by $\underline{{\mathbf i}}M^{\cdot}$ and called a $K$-injective or  fibrant resolution of $M^{\cdot}$.\\
Symmetrically, for every complex $M^{\cdot}\in \D(R)$ there is a quasi isomorphism $P^{\cdot}\to M^{\cdot}$ where  $ P^{\cdot}$ is a complex with projective terms. $ P^{\cdot}$ is also denoted by $\underline{{\mathbf p}}M^{\cdot}$ and called a $K$-projective or  cofibrant resolution of $M^{\cdot}$.

\item (\cite[Theorem 3.2 (b)]{Ke} and B\"okstedt and Neeman~\cite{BN}) Every additive functor $F$ defined on the module category $\Mod R$ admits a total right derived functor $\R F$ and a total left derived functor $\LL F$ defined on $\D(R)$.\\
Moreover, if $M^{\cdot}$ is a complex in $\D(R)$, then $\R  F (M^{\cdot})= F(\underline{{\mathbf i}}M^{\cdot})$ and $\LL  F (M^{\cdot})= F(\underline{{\mathbf p}}M^{\cdot})$. (We denote by $F$ also the functor induced on the homotopy category.)
\item (\cite[Theorem 3.2 (c)]{Ke})
If $T$ is an $S$-$R$-bimodule, then the adjoint pair $(G,H)$ of functors given by:
\[{H=\Hom_R(T,-)\colon \Mod R}\doublearrow{}{}{\Mod S\colon G=-\otimes _ST}\]
induces an adjoint pair of total derived functors
\[{ \R H=\R \Hom_R(T,-)\colon \D(R)}\doublearrow{}{}{ \D( S)\colon \LL G=-\overset{\LL}\otimes _ST}
\]

\end{itemize}
\begin{Thm}\label{T:derived-equivalence} Let $T_R$ be a right $1$-tilting module as in Notation~\ref{N:notation} and with endomorphism ring $S$.
 The following hold:
\begin{enumerate}
%\item[(1)] The adjoint pair $(G, H)$ induces an adjoint pair of total derived functors $\LL G,  \R H$ between $\D(S)$ and $\D(R)$.
\item[(1)] The counit adjunction morphism 
\[\eta\colon\LL G\circ \R H\to Id_{\D(R)}\] 
is invertible.
\item[(2)] The functor  $\R H\colon \D(R)\to \D(S)$ is fully faithful.

\item[(3)] There is a triangle equivalence $\Theta \colon \D(S)[\Sigma^{-1}]\to \D(R)$ such that $\LL G=\Theta\circ q$ where $q$ is the canonical quotient functor $q\colon \D(S)\to \D(S)[\Sigma^{-1}]$.
\item[(4)] If $\Sigma$ is the system of morphisms $u\in \D(S)$ such that $\LL G u$ is invertible in $\D(R)$, then $\Sigma$ admits a calculus of left fractions and the category $\D(S)[\Sigma^{-1}]$ coincides with the quotient category $\D(S)$ modulo the full triangulated subcategory $\Ker (\LL G)$ of the objects annihilated by the functor $\LL G$.
\end{enumerate}
\end{Thm}

We first prove condition (3) of Theorem~\ref{T:derived-equivalence} by a lemma.
\begin{Lem}\label{L:counit} In the assumptions of Theorem~\ref{T:derived-equivalence}, the counit adjunction morphism 
\[\eta\colon\LL G\circ \R H\to Id_{\D(R)}\] 
is invertible.
\end{Lem}
  \begin{proof} 
Let $M^{\cdot}$ be a complex in $\D(R)$ and consider a $K$-injective resolution $\underline{{\mathbf i}}M^{\cdot}$ of $M^{\cdot}$.
We have:
\[{\rm(1)}\qquad \R H (M^{\cdot})= \R \Hom_R(T, -)(M^{\cdot})=H(\underline{{\mathbf i}}M^{\cdot}).\]

Let $C^{\cdot}=H({\mathbf i}M^{\cdot})$.  $C^{\cdot}$ is a complex of right $S$-modules and 
\[\LL G (C^{\cdot})= \LL(-\otimes_S T)(C^{\cdot})= G(\underline{{\mathbf p}}C^{\cdot})\]
where $\underline{{\mathbf p}}C^{\cdot}$ is a $K$-projective resolution of $C^{\cdot}$ as a complex in $\D(S)$.

Consider the complex $T^{\cdot}:0\to {}_ST\to 0$ concentrated in degree $0$. A $K$-projective resolution $\underline{{\mathbf p}}T^{\cdot}$ of $T^{\cdot}$ in $\D(S)$ is the complex $0\to I\overset{\delta}\to S\to 0$ (from the exact sequence (b) in Lemma~\ref{L:S-sequence}). 

From the quasi-isomorphism $\underline{{\mathbf p}}T^{\cdot}\to T^{\cdot}$ and $\underline{{\mathbf p}}C^{\cdot}\to C^{\cdot}$ we get the chain of quasi-isomorphisms:
\[G( \underline{{\mathbf p}}C^{\cdot})= \underline{{\mathbf p}}C^{\cdot}\otimes_ST\leftarrow \underline{{\mathbf p}}C^{\cdot}\otimes_S\underline{{\mathbf p}}T^{\cdot}\to C^{\cdot}\otimes_S \underline{{\mathbf p}}T^{\cdot}.\]

Thus, $\LL G (C^{\cdot})=C^{\cdot}\otimes_S \underline{{\mathbf p}}T^{\cdot}$ and this gives
\[\LL G (C^{\cdot})=C^{\cdot}\overset{\LL}\otimes_S T 
= \Cone(1\otimes\delta).\]

From the exact sequence (b) in Lemma~\ref{L:S-sequence} we obtain the exact sequence of complexes of right $R$-modules:
\[\Tor^S_1(C^{\cdot}, T)\to C^{\cdot}\otimes_SI\to C^{\cdot}\otimes_SS\to C^{\cdot}\otimes_ST\to 0,\]

Now recalling that $C^{\cdot}$ is the complex $\R H (M^{\cdot})=H(\underline{{\mathbf i}}M^{\cdot})$,  Proposition~\ref{P:Hom-Tor}~(2) yields that the complex $\Tor^S_1(C^{\cdot}, T)$ has zero terms, hence we have the short exact sequence of complexes of right $R$-modules:
\[{\rm (2)}\quad 0\to C^{\cdot}\otimes_SI\to C^{\cdot}\otimes_SS\to C^{\cdot}\otimes_ST\to 0,\]

From (2) we obtain the long exact sequence in cohomology:
\[{\rm(*)}\dots\to H^{n+1}(C^{\cdot}\otimes_ST)\to H^n(C^{\cdot}\otimes_SI)\to\]
\[\to H^n(C^{\cdot}\otimes_SS)\to H^n(C^{\cdot}\otimes_ST)\to \dots\]

%Now from the triangle 
% \[C^{\cdot}\otimes I\overset{1\otimes\delta}\to C^{\cdot}\otimes S\to  \Cone (1\otimes\delta)\to (C^{\cdot}\otimes I)[1]\]
we also have the exact sequence of complexes of right $R$-modules:

\[0\to C^{\cdot}\otimes S\to \Cone (1\otimes\delta)\to (C^{\cdot}\otimes I)[1]\to 0\] 

from which we get the long exact sequence 

\[[{\rm(**)}\dots\to H^{n+1}(\Cone (1\otimes\delta))\to H^n(C^{\cdot}\otimes_SI)\to
H^n(C^{\cdot}\otimes_SS)\to\]
\[\to H^n(\Cone (1\otimes\delta))\to C^{\cdot}\otimes_S I\to \dots\]

Now comparing (*) with (**) we conclude that, for every $n\in \N$
\[H^n(C^{\cdot}\otimes_ST)\cong  H^n(\Cone (1\otimes\delta))\cong  H^n(C^{\cdot}\overset{\LL}\otimes_ST))\]

Hence,
\begin{center} $\LL G(C^{\cdot})=C^{\cdot}\overset{\LL}\otimes_ST$
is quasi isomorphic to $C^{\cdot}\otimes_ST.$
\end{center}

Letting $I^{\cdot}= \underline{{\mathbf i}}M^{\cdot}$,  we have $C^{\cdot}=\Hom_R(T, I^{\cdot})$ and we have also the commutative diagram:

\[ \begin{CD}
\dots@>>>\Hom_R(T, I^n)\otimes_ST@>>>\Hom_R(T, I^{n+1})\otimes_ST@>>>\dots\\
@. @V{\nu}VV @VV{\nu}V@. \\
\dots@>>>I^n@>>>I^{n+1}@>>>\dots
 \end{CD}\]

where the vertical maps are canonical isomorphisms by Proposition~\ref{P:Hom-Tor}~(1), since  $I^{\cdot}$ is a complex of injective right $R$-modules, hence belonging to the tilting class $T^\perp$.

Hence, $\Hom_R(T, I^n)\otimes_ST$ and $I^{\cdot}$ are  canonically isomorphic as complexes of $R$-modules, so we have:
\[\LL G(\R H(M^{\cdot}))=H( \underline{{\mathbf i}}M^{\cdot})\overset{\LL}\otimes_S T\cong \underline{{\mathbf i}}M^{\cdot}\cong M^{\cdot}.\]

%Thus we can conclude that, for every complex $M^{\cdot}\in\D(R)$ we have:
% 
% \[ \LL(-\otimes_ST)\circ \R\Hom (T, -) (M^{\cdot})=[\Hom_R(T-)( I^{\cdot})]\overset{\LL}\otimes_S T\cong \Hom_R(T, I^{\cdot})\otimes T=  I^{\cdot}=  M^{\cdot},\]
%where $I^{\cdot}=\underline{{\mathbf i}}M^{\cdot}$ is a $K$-injective resolution of $M^{\cdot}$. 
\end{proof}

\begin{proof} of Theorem \ref{T:derived-equivalence} 

Condition (1) is proved by Lemma~\ref{L:counit} and the equivalence of (1) with the other conditions follows essentially by applying  \cite[Proposition 1.3]{GZ}.

To complete the proof we add only a few comments.

The equivalence $\Theta \colon \D(S)[\Sigma^{-1}]\to \D(R)$, guaranteed by \cite[Proposition 1.3]{GZ},  is a triangle equivalence, since $\LL G$ is a triangle functor and $q$ is the canonical localization functor, so that the triangles in $\D(S)[\Sigma^{-1}] $ are images of triangles in $\D(S)$.

The functor $\LL G=-\overset{\LL}\otimes_S T$ is a triangle functor, hence $\Ker (\LL G)$ is a full triangulated subcategory of $\D(S)$. 
It is well known that the quotient category $\D(S)/\Ker (\LL G)$ is the localization of $\D(S)$ at the multiplicative system $\Sigma$ given by the morphisms $u\in \D(S)$ such that there exists a trinagle:
\[K^{\cdot}\to M^{\cdot}\overset u\to N^{\cdot}\to K^{\cdot}[1]\]
where $K^{\cdot}\in \Ker (\LL G)$ and $ M^{\cdot}, N^{\cdot}\in \D(S)$.
Thus $\Sigma$ coincides with the systems of morphisms $u\in \D(S)$ such that $\LL G (u)$ is invertible in $\D(R)$.

\end{proof}
%%%%%%%%%%%%%%%%%%%%%%%%%%%%%%%%%%%%%%%%%%%%%%%%%%%%%%%%%%%%%%%%%%%%%%%%%%%%%%%%%%%%%%%%%%%%%%


\begin{thebibliography}{99}

\bibitem{ATT}
{\sc L. Angeleri H\"ugel, A. Tonolo, J. Trlifaj}, \textit{Tilting
preenvelopes and cotilting precovers}, Algebr. Represent.  Theory 4 (2001), 155--170.


\bibitem{B}  {\sc S. Bazzoni}, \textit{Cotilting modules are
pure-injective}, Proc. Amer. Math. Soc. (131) (2003),
3665-3672.

%
%\bibitem{B2}  S. Bazzoni, \textit{A characterization of $n$-cotilting and $n$-tilting modules}, J. Alg. \textbf{273} (2004), 359-372.

%\bibitem{BET}  S. Bazzoni, P. Eklof, J. Trlifaj, \textit{Tilting cotorsion pairs,} to appear in Bull. London Math.
%Soc.
\bibitem{BH}  {\sc S. Bazzoni, D. Herbera}, \textit{One dimensional tilting modules are of finite type}, Algebr. Represent. Theory 11 (2008), no.1, 43-61.
%\bibitem{BS} S. Bazzoni, J. \v S\v tov\'\i\v cek, \textit{All tilting modules are of finite type}, to appear in Proc. Am.
%Math. Soc.

\bibitem{BN}
{\sc M. B\"okstedt and A. Neeman}, \textit{Homotopy limits in triangulated categories}, Compositio Math. 86 (1993), 209Ð234.

\bibitem{BB}
{\sc S. Brenner, M. Butler}, \textit{Generalizations of the
Bernstein-Gelfand-Ponomarev reflection functors}, in Proc.  ICRA III
LNM 832, Springer (1980), 103--169.

\bibitem{CF}
{\sc R. R. Colby, K. R. Fuller}, \textit{Tilting, cotilting and serially
tilted rings}, Comm. Algebra  18(5) (1990), 1585-1615.

\bibitem{C}
{\sc R. Colpi}, \textit{Tilting in Grothendieck categories} Forum Math. 11(1999), 735-759.

\bibitem{CT} {\sc R. Colpi and J. Trlifaj}, \textit{ Tilting modules and tilting torsion theories},
J. Alg. \textbf{178} (1995), 614--634.

\bibitem{CDT}  {\sc R. Colpi, G. D'Este, A. Tonolo}, \textit{Quasi-tilting modules and
counter equivalences}, J. Alg. 191(1997), 461-494.

\bibitem{CTT}
{\sc R. Colpi, A. Tonolo, J. Trlifaj}, \textit{Partial cotilting modules and
the lattices induced by them}, Comm.  Algebra 25(10) (1997),
3225--3237.

 \bibitem{F1} {\sc A. Facchini}, \textit{Divisible modules over integral domains},  Ark. Mat.  26  (1988),  no. 1, 67--85.

%
\bibitem{F2} {\sc A. Facchini}, \textit{A tilting module over commutative integral domains}, Comm. Alg. {15}(11) (1987), 2235--2250.

\bibitem{F} {\sc L. Fuchs}, \textit{On divisible modules over domains}, Abelian Gropus and modules, Proc. of the Udine Conference, CISM Courses and Lectures 287, Springer-Verlag, Wien-New York, 1984, 341-356.


\bibitem{GZ} {\sc P. Gabriel, M. Zisman}, \textbf{Calculus of fractions and homotopy category},Ergenbnisse  {35}, Springer-Verlag 1967

\bibitem{H} {\sc D. Happel}, \textit{On the derived category of a finite-dimensional algebra}, Comment. Math.
Helv. 62 (1987), no. 3, 339Ð389.

\bibitem{HR}
D. Happel, C. Ringel, \textit{Tilted algebras}, Trans.  Amer.  Math.
Soc.  \textbf{215} (1976), 81--98.

\bibitem{HRS} {\sc D. Happel, I.Reiten, S. Smal\o}, \textbf{Tilting in abelian categories and quasitilted algebras}, Memoirs  Amer.  Math.
Soc.  \textbf{575} (1995).


\bibitem{Ke2} {\sc B. Keller}, \textit{Derived categories and their uses}, Chapter of the Handbook of algebra, Vol. 1, edited by M. Hazewinkel, Elsevier 1996.

\bibitem{Ke} {\sc B. Keller}, \textit{Derived categories and tilting}, Handbook of tilting Theory, Lecture Note Series, LMS \textbf{332} 2007, 49-97

\bibitem{S} {\sc N. Spaltenstein}, \textit{Resolution of unbounded complexes},
Compositio Mathematica \textbf{65} (1988), 121--154.
%
\bibitem{Bo} {\sc B. Stenstr\"om},\textbf{Rings of quotients}, Grundleheren der Math., 217, Springer-Verlag, (1975).

%\bibitem{St}  J. \v S\v tov\'\i\v cek, \textit{All cotilting modules are pure injective},  Proc. Am.
%Math. Soc.

%\bibitem{C}
%R. Colpi, \textit{Tilting modules and *-modules}, Comm. Algebra 
%21 (1993), 1095--1102.



\end{thebibliography}
\end{document}